\documentclass[sigconf]{acmart}

\AtBeginDocument{%
  \providecommand\BibTeX{{%
    \normalfont B\kern-0.5em{\scshape i\kern-0.25em b}\kern-0.8em\TeX}}}

\usepackage{amsmath}
\usepackage{amsthm}
\newtheorem{rem}{remark}

\begin{document}

\title{Unconstraint minimization of continuous convex functions. Application to LP}

\author{Beniamin Costandin}
\email{bcostandin@yahoo.com}
\affiliation{ \institution{Technical University of Cluj Napoca}}

\author{Marius Costandin}
\email{costandinmarius@gmail.com}
\affiliation{ \institution{General Digits}}

\author{Petru Dobra}
\email{petru.dobra@aut.utcluj.ro}
\affiliation{ \institution{Technical University of Cluj Napoca}}

\begin{abstract}
Our contribution in this paper is two folded. We consider first the case of linear programming with real coefficients and give a method which 
allows the computation of a new upper bound on the distance from the origin to a feasible point. Next we present an application of the ellipsoid
method to form a novel algorithm which allows one to search for the continuous convex functions minimum without the prior knowledge on the 
search radius. If stopped early, the proposed algorithm can give proofs that the optimal value has not been reached yet, hence the user can opt
for more iterations. However, if prior guarantees exist on the existence of an optimum point in a given ball, the algorithm is guaranteed to find it.
For such a case we prove a polynomial upper bound on the number of flops required to obtain the solution. The presented algorithm is then 
applied to linear programming. We further develop our algorithm and provide a method to answer the feasibility question of linear programs with real coefficients in a 
number of flops bounded above by a polynomial in the number of variables and the number of constraints. However, in case of feasibility, our method does not generate an actual feasible point. For obtaining such a point, we  have to make some assumptions on the subgradients of a certain function. 
\end{abstract}

%



\maketitle

\section{Introduction}

The paper is basically composed out of three parts: the first part introduces a novel approach to study the feasibility of linear
 programs (and convex optimization in general). The feasibility is asserted by analyzing the minimum of a certain function. 
 In this framework we are able to formulate a novel idea which allows one to obtain an upper bound on the magnitude of a 
 feasible point (if exists) for linear programs. The basic idea is that the function that we associate to the system of 
 inequalities is strictly positive for any unfeasible point and negative for any feasible point. Our reasoning is that if its 
 minimum is greater than zero, then the system is not feasible. We are next able to assert analyzing the function minimum that
 a feasible point has to be in a ball centered at origin and of some radius which depends on the problem data. For so 
 doing, we use some results from fixed point theory. However, our given proof is valid only for a smooth approximation of our function.
 More details
 can be obtain in the below lines. The full proof is left for future work since it seems to involve deeper results of fixed 
 point theory, which are not in the scope of this paper. 

Since the presented theory relies on the minimization of a certain 
 continuous (not necessarily differentialbe) function, we present in the subsequent section an algorithm based on the ellipsoid 
 method. Our formulation is adapted to the "needs" of our method for linear programming presented earlier, although it can be used to
  minimize any continuous convex function. We also present some 
 other general advantages and disadvantages 
 of our formulation. Still in the same section, a merge between a slow but steady converging method (the ellipsoid method) and a fast,
 not always converging method (the pattern search) is proposed. Basically this would be a deep-cut ellipsoid method where the depth 
 of the cut is decided after a computationally cheap exploratory step. 
Finally we present an application of our proposed minimization algorithm to the feasibility theory linear programs. Here is where we 
obtain an upper bound on the search radius. This upper bound depends on the problem data (with real coefficients). 
The next section are the conclusions and future work.   

\section{Main results}
We give two results in this paper: a presentation of a feasibility study method for linear programs and an algorithm for minimizing
 continuous convex functions. We combine the two in the next section regarding applications. In case the reader is interested in the
minimization algorithm only, he (she) can safely skip the following subsection. 
\subsection{Feasibility of linear programs}
In this subsection we give a short algorithm which can be used to solve linear programs. The algorithm requires minimization of a 
continuous convex function. Let us consider the feasibility problem of a system of $m \in \mathbb{N}$ linear inequalities. We are interested
 if $\exists X \in \mathbb{R}^{n}$ such that 
\begin{align}\label{E2.1}
A_k^T \cdot X + b_k \leq 0 \hspace{0.5cm} \forall k \in \{1, \hdots, m\}
\end{align} where $A_k \in \mathbb{R}^{n \times 1}$ with $\|A_k\| > 0$ and $b_k \in \mathbb{R}$. We write (\ref{E2.1}) as 
\begin{align}
\begin{bmatrix} A_1^T \\ \vdots \\ A_m^T\end{bmatrix} \cdot X + \begin{bmatrix}b_1 \\ \vdots \\ b_m\end{bmatrix} &\preceq 
\begin{bmatrix}0 \\ \vdots \\0 \end{bmatrix} \nonumber \\
A \cdot X + B &\preceq 0
\end{align} Let us denote with $\mathcal{P}$ the set of all feasible points and define the function:

\begin{align}
f(X) = \max_{k \in \{1, \hdots, m\}} \left( A_k^T \cdot X + b_k \right)
\end{align} Let us now acknowledge some of the properties of the function $f$
\begin{enumerate}
\item It is convex: indeed, let $X,Y \in \mathbb{R}^n$ and $\alpha, \beta > 0$ $\alpha + \beta = 1$ then
\begin{align}
f(\alpha \cdot X + \beta \cdot Y ) &= \max_{k \in \{1, \hdots, m\}} A_k^T \cdot (\alpha \cdot X + \beta \cdot Y) + b_k \nonumber \\
& = \max_{k \in \{1, \hdots, m\}} \alpha \cdot \left (A_k^T \cdot X   + b_k  \right) + \beta \cdot \left( A_k^T \cdot Y + b_k \right) \nonumber \\
& \leq \max_{k \in \{1, \hdots, m\}} \alpha \cdot \left (A_k^T \cdot X   + b_k  \right) + \nonumber \\
 & + \max_{k \in \{1, \hdots, m\}} \beta \cdot \left (A_k^T \cdot X   + b_k  \right) \nonumber \\
& \leq \alpha \cdot f(X) + \beta \cdot f(Y)
\end{align}
\item It is continuous: indeed it can be proven that given two continuous functions $h,g : \mathbb{R}^n \to \mathbb{R}$ the function
\begin{align} 
f(X) = \max \{g(X), h(X) \}
\end{align} is continuous by simply observing that 
\begin{align}
f &= g + \max \{ h-g, 0\} \nonumber \\
& =  g + \frac{h-g + |h-g|}{2} = \frac{g + h + |g - h|}{2}
\end{align} is continuous. Next one can use the above to prove that 
\begin{align}
f(X) &= \max_{k} \{ f_1(X), \hdots, f_m(X)\} \nonumber \\
&=  \max \left\{ \max_k \{ f_1(X), \hdots, f_{m-1}(X)\}, f_m(X) \right\}
\end{align} is continous given the fact that $f_k$ is so.
\item It has strictly positive lower bounds on the magnitude of the subgradient at any (non-feasible) point. More on this in the next section. 
\end{enumerate} 

Next, let us define
\begin{align}
X^{\star} = \mathop{\text{argmin}}_{X \in \mathbb{R}^n} \hspace{0.1cm} f(X)
\end{align} We have the following observations regarding the above:

\begin{lemma}\label{L1.1}
The following is true:
\begin{align}
\mathcal{P} \neq \emptyset \iff f(X^{\star}) \leq 0
\end{align} 
\end{lemma}
\begin{proof} 
First let us assume that $\mathcal{P} \neq \emptyset$, than for all $X \in \mathcal{P} $ follows $f(X) \leq 0$, hence $f(X^{\star}) \leq 0$.
Next, assume that $f(X^{\star}) \leq 0$ hence $A_k^T\cdot X^{\star} +b_k \leq 0$, that is $X^{\star} \in \mathcal{P}$.
\end{proof}

\begin{rem} A similar Lemma to \ref{L1.1} can be given to characterize the intersection of convex sets.
\end{rem}

\begin{rem} In order to assert the feasibility of $\mathcal{P}$ is it enough to obtain a point $X$ for which $f(X) \leq 0$, hence the same result, 
as in Lemma \ref{L1.1}, can be obtained while working with the function $\underline{f}(X) = \max \{f(X), 0\}$. 
\end{rem}

While attempting to find $X^{\star}$, one has to start with an initial point, $X_0$. In case $X_0 = 0$ then $f(X_0) \leq \| B \|$. Moreover, 
we know that the norm of any subgradient of $f$ is bounded below by some strictly positive quantity if $f(\cdot) > 0$. We want to use this to obtain an initial
radius $R$ such that $X^{\star} \in B(0,R)$. For this we give the following two lemmas:
\begin{lemma} \label{L1.2}
Let $f:\mathbb{R}^n \to \mathbb{R}$ be a continuous convex function and $B(X_0,R)$ be a ball in $\mathbb{R}^n$ centered at $X_0 \in \mathbb{R}^n$
 and of radius $R > 0$. Then if $\exists \underline{d} > 0$ such that $\forall X \in B(X_0,R)$ one has $\| d_X\| \geq \underline{d}$
for all $ d_X \in \partial f (X)$ then $\exists X \in B(X_0, R)$ and  $ d_X \in \partial f(X)$ such that
\begin{align}
X - X_0  = -\frac{R}{2}  \cdot  \frac{d_X}{\| d_X \|}
\end{align}   
\end{lemma}

\begin{proof} Sketch: we will only give a proof for the easier case when $X_0 = 0$, $R = 1$ and $f$ has \textbf{continuous gradients} on $B[0,1]$. Let
$T : B[0,1] \to B[0,1]$ with 
\begin{align}
T(X) = -\frac{\nabla f(X)}{\| \nabla f(X)\|} \cdot \frac{1}{2} \in B[0,1]
\end{align} therefore, according to Brouwer's fixed point theorem, $\exists X $ such that 
\begin{align}
X = -\frac{\nabla f(X)}{ \| \nabla f (X)\|} \frac{1}{2}
\end{align} 

For the more general case one can probably work with the fact that the non-differentiable $f$ can be approximated by a differentiable one 
by some smoothing techniques. This is however, left for future work. Alternatively, the application of Kakutani fixed-point theorem should also
be considered.
\end{proof}

\begin{lemma}\label{L1.3}
Let $X_0 \in \mathbb{R}^n$ and $R > 0$ such that for all $X \in B(X_0,R)$ one has 
\begin{enumerate}
\item $0 \not\in \partial f(X)$. 
\item $\forall d_X \in \partial f (X)$ one has $\|d_X\| \geq \underline{d} > 0$
\end{enumerate} Then $ \exists X \in B(0,R)$
\begin{align}
f(X) \leq f(X_0) - \underline{d} \cdot \frac{R}{2}
\end{align}
\end{lemma}
\begin{proof}
Let $X,Y \in \mathbb{R}^n$ and let $d_X \in \partial f (X)$. Then
\begin{align}
 f(Y) \geq f(X) + d_X^T \cdot (Y - X)
\end{align} Let $Y = X_0$ to obtain:
\begin{align}
f(X) \leq f(X_0) + d_X^T \cdot (X - X_0)
\end{align}
According to Lemma \ref{L1.2} $\exists X \in B[X_0,R]$ and
 $d_X \in \partial f(X)$ such that $X - X_0 = -\frac{d_X}{\|d_X\|} \cdot \frac{R}{2}$ hence 
\begin{align}
f(X) & \leq f(X_0) - \frac{R}{2} \cdot \frac{d_X^T \cdot d_X}{ \| d_X \|} \nonumber \\
& \leq f(X_0) - \underline{d} \cdot \frac{R}{2}
\end{align}  
\end{proof}  

\begin{rem} \label{R3} 
According to Lemma \ref{L1.3}, if $\mathcal{P} \neq \emptyset$ then 
$B\left(0, \frac{2 \cdot \| B\|}{\underline{d}} \right) \cap \mathcal{P} \neq \emptyset$. Indeed the following two alternatives exist:
\begin{enumerate}  
\item $\exists X \in B\left(0, \frac{2\cdot \|B\|}{\underline{d}} \right)$ such that $0 \in \partial f(X)$, in which case
if $\mathcal{P}$ is feasible follows that $X \in \mathcal{P}$. 
\item $\not \exists X \in B\left(0, \frac{2\cdot \|B\|}{\underline{d}} \right)$ such that $0 \in \partial f(X)$. Therefore
using Lemma \ref{L1.3} one obtains that $\exists X \in B\left(0, \frac{2\cdot \|B\|}{\underline{d}} \right)$ such that
\begin{align} 
 f(X) \leq f(0) - \frac{\underline{d} \cdot R}{2} \leq \| B\| - \|B\| = 0
\end{align} hence $X \in \mathcal{P}$
\end{enumerate} 
\end{rem}

In case the reader is interested in the linear programming algorithm only, he (she) can skip the following four subsections. 
\subsection{Minimization of continuous convex functions}
In this section we will focus on finding the unconstraint minimum of continuous convex functions. In case more information about the 
function is known, like
\begin{enumerate} 
\item a minimizer exists
\item a lower bound on the minimum value 
\item a maximum distance to the global optimizer from a given initial point
\end{enumerate} we are able to configure the algorithm to converge in one metastep. We give complexity results for such a case. 
However, for the situations were such assumptions cannot be made about the function to be minimized the proposed algorithm can still be used
to search for the minimum in any desired/affordable number of metasteps. Each metastep has a finite runtime and a known worst 
case complexity. The algorithm is able to recognize if one of the metasteps terminated with the finding of the global minimum and stops. 
Otherwise, if stopped after a number of metasteps, the algorithm can give a certificate that the global minimizer has not been found.   

\begin{rem} Please note that the requirement of a lower bound and a maximum distance to the global optimizer is more general than the requirement
of the maximum distance to the global optimizer and Lipschitz continuity of $f$. Indeed, given $X_0$ and $\| X_0 - X^{\star}\| \leq R$ one already
 obtains $ | f(X^{\star}) - f(X_0)| \leq G \cdot |X_0 - X^{\star}|$ hence $-G\cdot R + f(X_0) \leq f(X^{\star})$ therefore $f$ is bounded below.
Therefore, in this regard, our algorithm is superior to those presented in \cite{lecS}, \cite{lecM}.
\end{rem}

In the following we present a series of simple lemmas that will be used later in the section.

\begin{lemma} \label{L2.1}
Let $\mathcal{C}\subseteq \mathbb{R}^n$ be a convex set and $X,Y \in \mathbb{R}^n$ with $X \in \mathcal{C}$. We consider the halfspace 
$\mathcal{H}(Y, H)$ for some $H \in \mathbb{R}^n$. Then 
\begin{align}
\mathcal{C} \cap \mathcal{H}(Y,H) \neq \emptyset \iff \left( \mathcal{C} \cap \mathcal{H}(X,H)\right) \cap \mathcal{H}(Y,H) \neq \emptyset
\end{align}
\end{lemma} 
\begin{proof} First let us notice that 
\begin{align}
\mathcal{C} = \left(\mathcal{C} \cap \mathcal{H}(X,-H) \right) \bigcup \left(\mathcal{C} \cap \mathcal{H}(X,H) \right)
\end{align}

We distinguish two cases:
\begin{enumerate}
\item $X \in \mathcal{H}(Y,H)$. It follows that $X \in \left( \mathcal{C} \cap \mathcal{H}(X,H) \cap \mathcal{H}(Y,H) \right) \neq \emptyset$ The 
conclusion is obvious here.
\item $X \not\in \mathcal{H}(Y,H)$ hence $H^T \cdot (X - Y) < 0$. Let $Z \in \mathcal{C} \cap \mathcal{H}(X, -H)$. Then 
\begin{align}
(Z - Y)^T\cdot H = (Z - X)^T \cdot H + (X - Y)^T \cdot H < 0
\end{align} therefore $\left( \mathcal{C} \cap \mathcal{H}(X, -H) \right) \cap \mathcal{H}(Y,H) = \emptyset$ hence the conclusion follows.
\end{enumerate}
\end{proof}

\begin{lemma} \label{L2.3}
Let $f : \mathbb{R}^n \to \mathbb{R}$ be a continuous convex function and $\mathcal{C} \subseteq \mathbb{R}^n$ be a convex set. We denote  
\begin{align}
X^{\star} = \mathop{\text{argmin}}_{X \in \mathcal{C}} \hspace{0.1cm} f(X)
\end{align} and assume that $X^{\star} \in \text{int} (\mathcal{C})$. Next let $\mathcal{D} \subseteq \mathcal{C}$ be a convex subset and
\begin{align}
Y^{\star} = \mathop{\text{argmin}}_{Y \in \mathcal{D}} \hspace{0.1cm} f(Y)
\end{align} If $Y^{\star} \in \text{int} (\mathcal{D}) $ then $f(Y^{\star}) = f(X^{\star})$
\end{lemma}
\begin{proof} Since $f$ is convex it has only one global minimum value. If $Y^{\star} \in \text{int} (\mathcal{D} )$ then 
$\exists \mathcal{B}(Y^{\star}, \epsilon)$ such that $f(Y) \geq f(Y^{\star})$ for all $Y$ in such a ball, hence
the conclusion follows.
\end{proof}

\begin{lemma} Let 
\begin{align}
\mathcal{E} &= \left \{  X \in \mathbb{R}^n | (X - X_c)^T \cdot A^T \cdot A \cdot (X - X_c) \leq 1\right \} \nonumber \\
		&= \left\{ A^{-1}\cdot U + X_c | \| U \| \leq 1 \right\}  
\end{align} denote an ellipsoid, and $\mathcal{H}(X_0, H) = \{ X | H^T \cdot (X-X_0) \leq 0\}$ denote a half space. Then 
\begin{align} 
\mathcal{E} \cap \mathcal{H} \neq \emptyset \iff \frac{| H^T \cdot (X_c - X_0)|}{\| H^T \cdot A^{-1}\|} < 1
\end{align}
\end{lemma}
\begin{proof}
Let us consider the transform: $ T(X) = A^T \cdot (X - X_c)$ then
\begin{align}
T(\mathcal{E}) &= \left\{ U | \| U \| \leq 1\right\} \nonumber \\
 T(\mathcal{H}) &= \left\{ A (X - X_c) | H^T X - H^T X_0 \leq 0\right\} 
\nonumber \\
& = \left\{ Y | H^T \cdot A^{-1}\cdot Y + H^T (X_c - X_0) \leq 0 \right \}
\end{align} Since if $X \in \mathcal{E} \cap \mathcal{H} \Rightarrow T(X) \in T(\mathcal{E}) \cap T(\mathcal{H})$ we obtain
$\mathcal{E} \cap \mathcal{H} \neq \emptyset \iff T(\mathcal{E}) \cap T(\mathcal{H}) \neq \emptyset$. But checking the later is
easy since is boild down to computing the distance from the origin to a hyperplane and compare it to $1$. The distance is
\begin{align}
 \frac{| H^T \cdot (X_c - X_0)|}{\| H^T \cdot A^{-1}\|} 
\end{align}
\end{proof}

\begin{definition}
For $f : \mathbb{R}^n \to \mathbb{R}$ continuous and convex and $\mathcal{C} \subseteq \mathbb{R}^n$ a convex set, we denote by
\begin{align}
\textbf{epi}(f) = \left\{ \begin{bmatrix}X, y \end{bmatrix}^T \in \mathbb{R}^{n+1} | \hspace{0.1cm}f(X) - y \leq 0 \right\}
\end{align} the epigraph of $f$. 
\end{definition} It is easy to see that $\textbf{epi}(f)$ is a convex subset of $\mathbb{R}^{n+1}$.

\subsection{Algorithm description}
In this subsection we give a verbose description of the idea of our algorithm. The below steps are perform by any metastep of 
the algorithm. The metastep has a known, deterministic, complexity. Upon exiting one metastep another metastep can be initiated if 
the optimum has not been found yet and the maximum number of metasteps hasn't been reached yet. Each metastep is guaranteed to 
decrease the function. Below is therefore the description of one metastep: 
\begin{enumerate}
\item Let $X_0 \in \mathbb{R}^n$
\item Consider the set $\mathcal{D} = \mathcal{B}\left(\begin{bmatrix}X_0 & f(X_0)\end{bmatrix}^T, R \right) \cap \textbf{epi}(f)$
\item Consider the set $\mathcal{G}(\alpha) = \left\{ \begin{bmatrix} X, y\end{bmatrix} \in \mathbb{R}^{n+1} | y \leq \alpha \right\}$
\item Consider the set $\mathcal{S}(\alpha) = \mathcal{D} \cap \mathcal{G}(\alpha) $
\item Find the smallest $\alpha$ such that $\mathcal{S} \neq \emptyset$. 
\item Let $\alpha^{\star}$ be the smallest such $\alpha$. Then let $X_{\alpha}^{\star} \in \mathcal{B}(X_0,R)$ such that 
$\begin{bmatrix} X_{\alpha}^{\star}, \alpha^{\star} \end{bmatrix} \in \mathcal{S}$. 
\item If $X^{\star}_{\alpha} \in \text{int} \left( \mathcal{B}(X_0,R) \right) $ then according to Lemma \ref{L2.3} we declare 
$X_{\alpha}^{\star}$ the global optimum point and $\alpha^{\star}$ the global optimum value of $f$.
\item Otherwise, if $X^{\star}_{\alpha} \in \partial \mathcal{B}(X_0,R) $ then we cannot say for sure if $X^{\star}_{\alpha}$ is the 
optimum point. We conclude that the optimum point has not been found (although it still can be, but one cannot be sure unless the
 previous step is chosen)
\end{enumerate}

We will only give details in what follows about the most important among the above steps: 
" Find the smallest $\alpha$ such that $\mathcal{S} \neq \emptyset$." For this we propose the following modification of the 
ellipsoid algorithm: 
\begin{enumerate}
\item It is easy to see that $ \begin{bmatrix}X_0, f(X_0) \end{bmatrix}^T \in \mathcal{G}(f(X_0) + R) \cap \mathcal{D} \neq 
\emptyset$ while $\mathcal{G}(f(X_0)-R) \cap \mathcal{D} = \emptyset$. Indeed, for the later assume that 
$\exists \begin{bmatrix} X, y\end{bmatrix}^T \in \mathcal{D}$ such that $f(X) \leq y \leq f(X_0) - R$ therefore $R \leq f(X_0) - f(X)$.
However, because $\begin{bmatrix} X, y\end{bmatrix}^T \in \mathcal{D}$ follows that $\|X - X_0\|^2 + (f(X) - f(X_0))^2 \leq R^2$ hence
$ \| X - X_0\| = 0$. But this contradicts the fact that $f(X) \neq f(X_0)$, since $f$ being a function associates exactly one value to
any argument. 
\item Let $\bar{\alpha} = f(X_0) + R$ and $\underline{\alpha} = f(X_0)-R$. We proved that $\mathcal{S}(\bar{\alpha}) \neq \emptyset$ while
$ \mathcal{S}(\underline{\alpha}) = \emptyset$. It is now possible to find the smallest $\alpha$ such that $\mathcal{S}(\alpha)$ is feasible
by simply bisecting the interval $[\underline{\alpha}, \bar{\alpha}]$. 
\end{enumerate}

We will focus here on asserting if for a given $\alpha \in \mathbb{R}$ one has $\mathcal{S}(\alpha) = \emptyset$ or not.
We use the ellipsoid algorithm to assert the feasibility of $\mathcal{S}(\alpha)$

\begin{enumerate}
\item Consider the inital ellipsoid $\mathcal{E}_0 = \mathcal{B}\left(\begin{bmatrix}X_0 & f(X_0)\end{bmatrix}^T, R \right)$. \\ 
Obviously $ \mathcal{D} \subseteq \mathcal{E}_0 $. 
\item Let $\mathcal{E}_k$ be ellipsoid at step $k$ and let $\begin{bmatrix} X_k, y_k\end{bmatrix}$ denote the center of the ellipsoid.
 We consider three possible alternatives:
\begin{itemize}
\item $\begin{bmatrix} X_k, y_k\end{bmatrix}^T \in \mathcal{S} = \mathcal{D} \cap \mathcal{G}$ then stop and declare $\mathcal{S}$ feasible.
\item $\begin{bmatrix} X_k, y_k\end{bmatrix}^T \in \mathcal{D}(\alpha)$ but $\begin{bmatrix} X_k, y_k\end{bmatrix}^T \not\in \mathcal{G}(\alpha)$.
 Then according to Lemma \ref{L2.1} it is enough to cut the ellipsoid with the hyperplane $\{ [X, y] \in \mathbb{R}^{n+1}| y \leq y_k\}$ and 
therefore obtain the next smallest volume ellipsoid containing $\mathcal{E}_k \cap \{ [X, y] \in \mathbb{R}^{n+1}| y \leq y_k\}$. 
We have two conditions here for stopping the algorithm: 
\begin{enumerate}
\item if $\|y_k - \alpha\| \leq \epsilon$ stop and declare, with $\epsilon$ accuracy, $\mathcal{S}(\alpha)$ is feasible.
\item if $\mathcal{E}_k \cap \mathcal{G}(\alpha) = \emptyset$ stop and declare $\mathcal{S}(\alpha)$ infeasible. 
This can be easily tested according to ...
\end{enumerate} 
\item $\begin{bmatrix} X_k, y_k\end{bmatrix}^T \not \in \mathcal{D}(\alpha)$ Then we need to find $E \in \mathbb{R}^{n+1}$ such that 
$E^T \cdot \begin{bmatrix} X-X_k &y-y_k\end{bmatrix}^T \leq 0$ for all $\begin{bmatrix}X &y \end{bmatrix}^T \in \mathcal{D}(\alpha)$,
 that is a separating hyperlane. Please note that because $f$ is a convex function, follows that 
$f(X) \geq f(X_k) + d_{X_k}^T \cdot (X - X_k)$ where $d_{X_k} \in \partial f(X_k)$, hence \\
$\begin{bmatrix}d_X^T &-1 \end{bmatrix} \cdot \begin{bmatrix} X-X_k &f(X) - f(X_k) \end{bmatrix} \leq 0$ hence
$\begin{bmatrix}d_X^T &-1 \end{bmatrix} \cdot \begin{bmatrix} X-X_k &y - f(X_k) \end{bmatrix} \leq 0$ for all
$[X,y] \in \textbf{epi}(f)$. Therefore, let $E^T = [d_{X_k}^T, -1]^T$ be the separating hyperplane and now one can iterate to 
obtain the next minimum volume ellipsoid $\mathcal{E}_{k+1}$. One can stop the algorithm, with the conclusion that $\mathcal{S}(\alpha)$ 
is infeasible, if $\mathcal{E}_{k+1} \cap \{[X,y]| y \leq \alpha\} = \emptyset$
\end{itemize}
\end{enumerate}
\begin{rem} Because $\mathcal{D} \cap \mathcal{E}_k$ is always feasible, we need stopping criteria only for the case when the center
of the ellipsoid is in $\mathcal{D}$
\end{rem}
\begin{rem} If none of the above stopping criteria in not met, ultimately one ca rely on the standard stopping criteria: the volume of 
the ellipsoid is smaller than a certain value. However, one can eventually attempt to prove with rigor that in the presented situation  
this is not the happening. One can eventually use the fact that the center of the ellipsoid is moving closer to the cutting hyperplane.
\end{rem}

\subsection{Complexity analysis}
Given a precision $\epsilon > 0$ we ask $\log_2\left( \frac{\bar{\alpha} - \underline{\alpha}}{\epsilon}\right)$ times if $\mathcal{S}(\alpha)$ is feasible
in order to find the minimum $\alpha^{\star}$ such that $\mathcal{S}(\alpha^{\star})$ is feasible. However, for answering the feasibility
question of $\mathcal{S}(\alpha)$ one has to apply the ellipsoid algorithm. 

Let us assume that $\exists \mathcal{B}(X^{\star}, \epsilon) \in \mathcal{S}(\alpha)$ and we started with the initial ellipsoid 
$\mathcal{B}(0,R)$. Then 
\begin{align}
\frac{\text{vol}\left( \mathcal{B}(X^{\star}, \epsilon) \right) }{\text{vol}\left( \mathcal{B}(X_0, R) \right) } \leq 
\left( \frac{\epsilon}{R}\right)^{n+1}
\end{align}

It is known that 
\begin{align}
\frac{\text{vol}(\mathcal{E}_k)}{\text{vol}(\mathcal{E}_0)} \leq e^{-\frac{k}{2(n+2)}}
\end{align} This means that if we want $\frac{\text{vol}(\mathcal{E}_k)}{\text{vol}(\mathcal{E}_0)} \leq 
\left( \frac{\epsilon}{R} \right)^{n+1}$ 
it is sufficient to have $e^{-\frac{k}{2(n+2)}} \leq \left( \frac{\epsilon}{R}\right)^{n+1}$ which means 
\begin{align}
k \geq 2\cdot (n+2)\cdot (n+1) \cdot \log \left(\frac{R}{\epsilon}\right)
\end{align}

Furthermore, each step in the ellipsoid algorithm requires $\mathcal{O}(n^2) \cdot \Omega$ flops, where $\Omega$ is the complexity required to provide a 
subgradient, hence the overall complexity of a metastep is
\begin{align}
\mathcal{O} \left( n^4 \cdot \log \left(\frac{R}{\epsilon} \right)^2 \cdot \Omega\right)
\end{align} if one takes $\bar{\alpha} = R$ and $\underline{\alpha} = -R$.

\begin{rem} The advantages of the algorithm presented in the previous subsections are costing an extra $\log\left( \frac{R}{\epsilon}\right)$
at the complexity of the metastep. This is acceptable, if the quotient of $R$ and $\epsilon$
is not very large, since it allows minimizing non-smooth continuous convex functions without the need for a priori bounds on the gradient of $f$ or
search radius. 
\end{rem}

A possible practical improvement is presented in the following subsection. 
\subsection{A merge between pattern search and deep cut ellipsoid}
In this subsection, we present a possible improvement for the ellipsoid algorithm. We observed in the previous sections that the large 
number of flops required to assert feasibility is due in part to the fact that the volume of the ellipsoid is decreasing slowly. 
A possible improvement here
is the so called deep cut ellipsoid algorithm. We give in the following a method to obtain a possible deep cut without much extra 
computation required per iteration. Assume that $X_k$ and $A_k$ describe the ellipsoid $\mathcal{E}_k$ at step $k$. 
 Assuming that we want to assert the existence of o point in the set 
$\{X | f(X) \leq 0\} \subseteq \mathcal{E}_k $ we can proceed as follows. From the current center of ellipsoid, $X_k$ we
can perform some, so called "exploratory moves/checks" in a predefined (fixed or updated) directions see \cite{patternS1}. Since we already have an 
ellipsoid one can think of using its principal axes for such checks. However, those are actually the eigenvectors of $A^T_k\cdot A_k$ and
their computation might not be justified since it increases the complexity of each step. We do not know at the moment a simple update
formula for the eigenvectors of rank one updated matrix. Nevertheless, one can always resort to simpler/fixed search directions.
Consider the columns of the unit matrix, 
$(e_p)_{p \in \{1, \hdots, n\}}$ as the "explore" directions. Therefore one checks the following points
\begin{align}
X_{k,p}^{\pm} = X_k \pm \beta \cdot e_p \hspace{1cm} p \in \{ 1, \hdots, n\}
\end{align} to see if $f(X_{k,p}^{\pm}) \leq \underline{f}_k$, where $\underline{f}_k$ is the lowest value of $f$ known. Of course
if $\underline{f}_k \leq 0$ the algorithm stops. 
 
According to
\cite{lecS} in such situations a "deeper" cut can be made. Indeed, let $d_{X_k} \in \partial f(X_k)$. Then 
\begin{align}
f(X) \geq f(X_k) + d_{X_k}^T \cdot (X - X_k)
\end{align} for all $X\in \mathbb{R}^n$. But because we search for $f(X) \leq 0$, hence we discard those $X$ for which $f(X) \geq \underline{f}_k$.
 Therefore for those $X$ that are of interest to us to obtain:
\begin{align}
\underline{f}_k \geq f(X) \geq f(X_k) + d_{X_k}^T\cdot (X-X_k) \iff \nonumber \\
d_{X_k}^T \cdot (X - X_k) + f(X_k) - \underline{f}_k \leq 0
\end{align} where $f(X_k) - \underline{f}_k \geq 0$. There is actually an improvement if the inequality is strict.  Details about the deeps cut formulas can be found in  \cite{stanf_ellipsoid}, \cite{lecS}, \cite{laszlo}, \cite{todd}.

\section{Application to LP}
Given a system of linear inequalities, as (\ref{E2.1}), with real coefficients it would be interesting to investigate the following:
\begin{enumerate} 
\item if it is possible to assert its feasibility in a number of flops bounded above by a polynomial in $m,n$
\item if feasible, find a solution (since answering the first question should not necessary mean that a solution was found! )
\end{enumerate}
\begin{rem} For the first question, in the case the coefficients are integers there is already well known the fact that assertion of the 
feasibility can be done in so called "weakly" polynomial time \cite{ref1_2011}, \cite{lecCor}, i.e the polynomial bounding above the number for operations depends on the 
problem data.  
\end{rem}
\subsection{Deciding feasibility}
For our first question, let us recall the so called Farkas lemma:
\begin{lemma}
Let us consider the linear system of inequalities:
\begin{align}
A \cdot X + B \prec 0
\end{align} where $A \in \mathbb{R}^{m\times n}$, $X \in \mathbb{R}^n$ and $B \in \mathbb{R}^m$ with $m\geq n$ positive integers. Then exactly
one of the following two is true:
\begin{enumerate}
\item $\exists X \in \mathbb{R}^n$ such that $A \cdot X + B \preceq 0$
\item $\exists y \in \mathbb{R}^m_{+}$ such that $y^T\cdot A = 0$ and $B^T \cdot y > 0$ 
\end{enumerate}
\end{lemma}

We define the following convex optimization problem:
\begin{align}\label{E36}
d^{\star} &= \min q^T \cdot A \cdot A^T \cdot q \hspace{0.5cm} \text{s.t} \hspace{0.1cm}
\begin{cases} q \succeq 0\\ B^T \cdot q \geq 0\\ \textbf{1}^T \cdot q \geq 1\\ \textbf{1}^T \cdot q \leq 2\end{cases}
\end{align} The constraints are explained as follows: $q \succeq 0, B^T \cdot q > 0$ are required by Farkas lemma, $B^T \cdot q = 0$ was
added to avoid an open search space, $\textbf{1}^T \cdot q \geq 1$ was added to exclude the otherwise trivial solution $q = 0$ and finally
$\textbf{1}^T \cdot q \leq 2$ was added to make the search space bounded. 
Since the above is a convex quadratic minimization problem and the search space is closed, we can find a solution and denote it $q^{\star}$. 
We have the following alternatives:
\begin{enumerate}
\item $d^{\star} = 0$ In such a case $\exists q^{\star} \succ 0$ such that $(q^{\star})^T \cdot A = 0$. Then:
\begin{enumerate} 
\item If $B^T \cdot q^{\star} > 0$ follows that $A \cdot X + B \preceq 0$ has no solution 
\item If $B^T \cdot q^{\star} = 0$ follows that $A \cdot X + B \prec 0$ has no solution
\end{enumerate}
\item $d^{\star} > 0$ In this case follows that $\not \exists q \succeq 0$ such that $q^T \cdot B > 0$ and $q^T \cdot A = 0$ since 
in such a case $\frac{q}{1^T \cdot q} $ belongs to the search space and would vanish the objective function dropping it strictly below its 
computed minimum. From here it is concluded using Farkas's lemma, that $\exists X$ such that $A \cdot X + B \preceq 0$. 
\end{enumerate}

It is obvious that we can assert the feasibility of $A \cdot X + B \preceq 0$ (with the above details) just by solving the above convex
optimization problem. It is important to see that the search space is include in the ball $\mathcal{B}(0, 2 \cdot \sqrt{m})$ regardless 
of $A,B$. Using our above proposed algorithm (or the general ellipsoid method as well) one can obtain a polynomial upper bound 
on the number of flops, required to assert the feasibility of the linear program, with no dependency on the problem data. The problem data 
is, of course, involved (in one way or another) in the computation but is not affecting the number of flops. As we have shown above,
 the computation complexity for the ellipsoid method depends on $n,m,R,\epsilon$. In this case $R = 2 \cdot \sqrt{m}$. 
\subsection{Search for a feasible point}
For our second question, assuming the $LP$ is feasible we want to find a solution. We recall the results from Remark (\ref{R3}). Also, we will consider that $\|A_k\| \leq 1$ and $|b_k| \leq 1$. If this is not the case one can simply 
multiply the inequality by $\frac{1}{\sqrt{\|A_k\|^2 + b_k^2}}$. Therefore $\|B\| \leq \sqrt{m}$ and hence
\begin{align}
\mathcal{B}\left(0,\frac{2\cdot \sqrt{m}}{\underline{d}} \right) \bigcap \mathcal{P} \neq \emptyset
\end{align} 
Recall 
\begin{align}
f(X) = \max_{k \in \{1, \hdots, m\}} \left( A_k^T \cdot X + b_k \right)
\end{align}

Let for $X \in \mathbb{R}^n$ with $f(X) > 0$. Being interested in $\partial f(X)$ we define, \cite{lecSub}
\begin{align}
\mathcal{A}(X) &= \{ k \in \{1, \hdots, m\} | A_k^T \cdot X + b_k = f(X)\} \nonumber \\
\partial f(X) &= \left\{ \sum_{k \in \mathcal{A}(X)} \alpha_k \cdot A_k | \alpha_k \geq 0, \sum_{k \in \mathcal{A}(X)} \alpha_k = 1 \right \}
\end{align}

The main focus on the rest of the paper will be to investigate $\underline{d}$ i.e a lower bound on the magnitude of the subgradients for 
the points $X$ with $f(x) > 0$. We have a first result in this direction:
\begin{lemma} For a feasible system, if $f(X) > 0$, then $\| d_X\| > 0$ for all $d_X \in \partial f(X)$
\end{lemma}
\begin{proof}
 Indeed, let $f(X) > 0$ and assume that exists $\alpha_k \geq 0$ with $\sum_{k \in \mathcal{A}(X)} \alpha_k = 1$ such that 
\begin{align}
\sum_{k \in \mathcal{A}(X)} \alpha_k \cdot A_k = 0
\end{align} Then, $q_k = \begin{cases} \alpha_k \hspace{0.2cm} &k \in \mathcal{A}(X) \\ 0 & k \not \in \mathcal{A}(X)\end{cases}$ forms a certificate
of infeasiblity since
\begin{enumerate}
\item $q^T \cdot A = \sum_{k \in \mathcal{A}(X)} \alpha_k \cdot A_k = 0 $
\item $q^T \cdot B > 0$ since $f(X) > 0$
\begin{align}
0 &< \sum_{k=1}^m q_k \cdot (A_k^T \cdot X + b_k) = \sum_{k \in \mathcal{A}(X)}\alpha_k \cdot (A_k \cdot X + b_k) \nonumber \\
&= \sum_{k\in \mathcal{A}(X)} \alpha_k \cdot b_k = q^T \cdot B
\end{align}
\end{enumerate} However, this is a contradiction with the assumed feasibility of the system.  
\end{proof}

\begin{rem} The result above can be extended to the frontier of the feasible region if the linear system is assumed to be strictly feasible. We conclude that for all $X$ with $f(X) \geq 0$ $\exists \underline{d} > 0$ such that $\|d_X\| \geq \underline{d}$ 
\end{rem} 
\subsection{Search for a lower bound of the subgradients of infeasible points}
Unfortunately, we were not able so far to provide a computation method for the $0 < \underline{d}$, that is a strictly positive 
lower bound on the magnitude of the subgradients for the infeasible points (although we proved its existence). 
Having such a bound, one can use the results form the previous sections to provide an initial ball containing a solution. We will nevertheless 
give two methods which can be used in practice:
\begin{enumerate}
\item consider $0 < \epsilon \leq \underline{d}$ and search for a solution in $\mathcal{B} \left( 0, \frac{2\cdot \sqrt{m}}{\epsilon} \right)$. If a 
solution is found stop and return it, otherwise $\epsilon \gets \frac{\epsilon}{2}$ and reiterate. 
\item Let $\Lambda = \begin{bmatrix} \alpha_1 & \hdots 
&\alpha_m\end{bmatrix}^T \in \mathbb{R}^m$ and consider the following optimization problem for a $X$ with $f(X) \geq 0$  
\begin{align} \label{E42}
\underline{d}(X) = \min _{\begin{cases} \textbf{1}^T \cdot \Lambda = 1\\ \Lambda \succeq 0 \\ \Lambda^T \cdot (A \cdot X + B) \geq 0 \end{cases}} 
\Lambda^T \cdot A \cdot A^T \cdot \Lambda
\end{align}
Please note that for all $d_X \in \partial f(X)$ with $f(X) \geq 0$, we have 
$d_X = q_X \cdot A$ with obvious definition for $d_X$ and $q_X$. Moreover, $q_X \succeq 0$ and $q_X^T \cdot \textbf{1}^T = 1 $. Furthermore, because $0 \leq f(X) = q_X^T \cdot (A\cdot X + B)$, follows $\|d_X\| \geq \underline{d}(X)$.
The above, (\ref{E42}) is a quadratic minimization problem where $A \cdot A^T \in \mathbb{R}^{m \times m}$ is P.S.D hence the function can have 
zero values for $\Lambda \neq 0$. There are the following alternatives:
\begin{enumerate}
\item $\underline{d}(X) = 0$ then $\exists \Lambda^T \cdot A = 0$ with $0 \neq \Lambda \succeq 0$ therefore $\Lambda^T \cdot (A\cdot X + B) = \Lambda^T \cdot B \geq 0$ for $f(X) \geq 0$. It follows that the system is not strictly feasible which is a contradiction. 
\item $\underline{d}(X) > 0$ then this is actually a lower bound for all the subgradients of $f$ at the point $X$, i.e $\|d_X\| \geq \underline{d}(X)$   
\end{enumerate} 

We want to find $\underline{d}$ with $0<\underline{d} \leq \min_{\begin{cases}X \in \mathbb{R}^n\\  f(X) \geq 0 \end{cases}} \underline{d}(X)$. 
%
%
%

 Let us also define: 

\begin{align}
\underline{c} = \min _{\begin{cases} \textbf{1}^T \cdot \Lambda = 1\\ \Lambda \succeq 0 \end{cases}} 
\Lambda^T \cdot A \cdot A^T \cdot \Lambda
\end{align} Plese note that $\underline{c} \leq \underline{d}$. We have the following reasoning:
\begin{enumerate}
\item If $\underline{c} > 0$ take $\underline{d} = \underline{c}$ and conclude that the ball 
$\mathcal{B} \left(0, \frac{2\cdot \sqrt{m}}{\underline{d}} \right)$ has a feasible point.
\item If $\underline{c} = 0$ compute 
\begin{align}
\bar{b} = \max_{\begin{cases} \textbf{1}^T \cdot \Lambda = 1\\ \Lambda \succeq 0 \\ \Lambda^T \cdot A = 0 \end{cases}} 
B^T \cdot \Lambda
\end{align} Since the system is assumed to be strictly feasible, it follows from Farkas's lemma that $\bar{b} < 0$. This means that $\underline{a} = 0$
where 
\begin{align}
\underline{a}= \min _{\begin{cases} \textbf{1}^T \cdot \Lambda = 1\\ \Lambda \succeq 0 \\ \Lambda^T \cdot  B \leq -|\bar{b}| 
\end{cases}} 
\Lambda^T \cdot A \cdot A^T \cdot \Lambda
\end{align} and $\bar{a} > 0$ where 
\begin{align}
\bar{a} = \min _{\begin{cases} \textbf{1}^T \cdot \Lambda = 1\\ \Lambda \succeq 0 \\ \Lambda^T \cdot  B \geq -|\bar{b}| + \epsilon 
\end{cases}} 
\Lambda^T \cdot A \cdot A^T \cdot \Lambda
\end{align} 
\end{enumerate}
Take $\underline{d} = \bar{a}$
\end{enumerate}

%
\begin{rem}
The problem of finding $\underline{d}$ can be solved with the above proposed algorithm, after a proper forming of the epigraph. 
Moreover, the search space is limited to $B(0, \sqrt{m})$ since $0 \leq \alpha_k \leq 1$. There is therefore a polynomial (in $m$) bound 
on the number of flops required to find $\underline{d}$ regardless of $A,B$ since 
$\|A^T \cdot \Lambda\| \leq \| A\| \cdot \| \Lambda\|$ and the norm of $A$ can be 
controlled by some pre-processing.  For instance, if $\|A_k\| \leq 1$ then $\| A \|_F = Tr \left( A \cdot A^T \right) \leq 
\sum_{k=1}^m \|A_k\| \leq m$
\end{rem}

For the above proposed algorithm it is enough to take $R = \sqrt{\|B\|^2 + \left(\frac{\| B\|}{\underline{d}} \right)^2}$. 
Taking $ R \geq \sqrt{m} \cdot \sqrt{\frac{\underline{d}^2 + 1}{ \underline{d}^2}}$ is enough to be sure that the ball $B(0,R)$ contains
$\begin{bmatrix} X^{\star} \\f(X^{\star}) \end{bmatrix}$ where $f(X^{\star}) \leq 0$ hence $X^{\star} \in \mathcal{P}$. 
The feasibility of an LP can therefore be asserted within a number of flops bounded above by
\begin{align}
\mathcal{O} \left( n^4 \cdot \log \left( \frac{\sqrt{m}}{\epsilon} \cdot \sqrt{\frac{\underline{d}^2 + 1}{ \underline{d}^2}} \right)^2 \right) 
\end{align} where $\epsilon$ is the precision required for the ellipsoid algorithm and $\underline{d}$ is given above. 
%

\section{Conclusion and future work}
We presented in this paper two important topics namely minimization of continuous convex functions and we gave a new theory for the study of the feasibility of linear programs with real coefficients, \cite{ref3_1990}. There are some beautiful results, like asserting the feasibility in strongly polynomial time, but unfortunately, there are also some gaps in the presented work, especially related to the theory of linear programs. First, for future work remains the completion of the proof for our first result regarding the initial radius. Our proof uses Brouwers fixed point theorem and it is valid therefore only for the gradients of continuous differentiable functions. However, the gradients of the function we are using are not continuous (but can be approximated by the gradients of a continuous differentiable function). This proof has to be completed with the necessary rigor. Next, we were not able to give a method for bounding below the magnitude of the subgradients of infeasible points, and we had to use some assumptions involving $\epsilon$. As a future work this should be improved with an algorithm which provides an proper computation method for the lower bound of the subgradients of infeasible points and hence generates a ball to initialize the ellipsoid algorithm with, when searching for a feasible point.

\begin{acks}
This paper was supported by the project "Entrepreneurial competences and excellence research in doctoral and postdoctoral programs-ANTREDOC" project co-founded by the European Social Fund
\end{acks}

\bibliographystyle{ACM-Reference-Format}
\bibliography{sample-base}
\appendix

\section{Overview of the ellipsoid algorithm}

Let $S^{n}_{++}$ be the set of all positive definite symmetric matrices in $\mathbb{R}^n$, $A \in S^{n}_{++}$ and $X_c \in \mathbb{R}^n$. 
We will denote an ellipsoid in the following manner:
\begin{align}
\mathcal{E}(X_c, A) & = \left\{ X \in \mathbb{R}^n | (X - X_c)^T \cdot A^T\cdot A \cdot (X - X_c) \leq 1 \right\} \nonumber \\
& = \{ A^{-1}\cdot U + X_c | \| U \| \leq 1\}
\end{align}  We will also work with hyperplanes and halfspaces, hence given $H \in \mathbb{R}^n$ we denote by
\begin{align}
\mathcal{H}(X_c, H) = \left\{ X \in \mathbb{R}^n | H^T \cdot (X - X_c) \leq 0 \right\}
\end{align} 

Given $A_k \in \mathcal{S}^n_{++}, X_k, H \in \mathbb{R}^n$ a fundamental step in the ellipsoid algorithm requires the ability to find the smallest 
volume ellipsoid $\mathcal{E}(X_{k+1},A_{k+1})$ containing $\mathcal{E}(X_k, A_k) \cap \mathcal{H}(-H^T\cdot X_c, H)$. Iterative formulae for $A_{k+1}$ and 
$X_{k+1}$ as functions of $A_k$, $X_k$ and $H$ exist. Let 
\begin{align}
\mathcal{E}_k = \mathcal{E}(X_{k},A_{k}) & = \left\{ X |  \| A_{k} \cdot  (X - X_{k}) \| \leq 1\right\} \nonumber \\
& = \left \{ A_{k}^{-1} \cdot U + X_{k} | \|U\| \leq 1 \right \}
\end{align} 

\begin{align}
\mathcal{H}_k = \mathcal{H}(X_k,H_k) = \{ X | H^T_k \cdot (X - X_k) \leq 0 \}
\end{align}and consider the linear transform $T_1 : \mathbb{R}^n \to \mathbb{R}^n$ $T_1(X) = A_{k} \cdot (X - X_k)$ then 
\begin{align}
T_1(\mathcal{E}_{k}) = \left\{ U | \| U \| \leq 1 \right\} 
\end{align} the unit ball, and 
\begin{align}
T_1(\mathcal{H}_k) &= \left\{ A_k \cdot (X - X_c) | H^T \cdot (X - X_c) \leq 0 \right\} \nonumber \\
& = \{ Y | H^T \cdot A_k^{-1}\cdot Y \leq 0\} = \mathcal{H}(0, A_k^{-1}\cdot H_k)
\end{align} which is another halfspace. 
\begin{rem}
Under the $T_1$ transform $\mathcal{E}_k$ becomes the unit sphere, $B(0,1)$ and $\mathcal{H}_k$ becomes a halfspace through the origin, 
$\mathcal{H}(0,\hat{H}_k)$ 
\end{rem}
Next, the problem is further simplified by the use of $T_2: \mathbb{R}^n \to \mathbb{R}^n$, $T_2(X) = Q \cdot X$ where $Q \in \mathbb{R}^{n\times n}$ 
with $Q^T \cdot Q = I$ and $ Q\cdot \hat{H}_k = -\|\hat{H}_k\| \cdot e_1$, $e_1$ being the first column of the unit matrix $I$, to obtain
\begin{align}
T_2(B(0,1)) = B(0,1) \hspace{0.5cm} T_2(\mathcal{H}(0,\hat{H}_k)) &= \{ Q\cdot X | \hat{H}^T\cdot X \geq 0\} \nonumber \\
& = \{ X | e_1^T\cdot X \geq 0\} = \mathcal{H}(0,e_1)
\end{align} The solution to the simpler problem of obtaining the smallest volume ellipsoid containing $B(0,1) \cap \{ X| e_1^T \geq 0\}$ 
is known, see \cite{stanf_ellipsoid}, \cite{lecS}, \cite{capozzo}, \cite{lecCMU}, \cite{laszlo}, \cite{todd},  for a detailed derivation. \textbf{It is independent of the initial problem}. We call this ellipsoid $\mathcal{E}_U$
\begin{align}
\mathcal{E}_U &= \{ X | (X - m\cdot e_1)^T M^T \cdot M \cdot (X - m\cdot e_1) \leq 1\} \nonumber \\
& = \{ M^{-1} \cdot U + m\cdot e_1 | \|U\| \leq 1\}
\end{align} and 
\begin{align}
M^T\cdot M = \frac{n^2 - 1}{n^2} \cdot \left( I + \frac{2}{n-1} \cdot e_1 \cdot e_1^T \right) \hspace{1cm} m = \frac{1}{n+1}
\end{align} It is also proven that 
\begin{align}
\frac{\text{vol}(\mathcal{E}_U)}{\text{vol}(B(0,1))} \leq e^\frac{-1}{2(n+1)}
\end{align}
Now, the desired ellipsoid is $T_1^{-1} (T_2^{-1}(\mathcal{E}_U))$. Computing $T_2^{-1}(\mathcal{E}_U) = Q^{-1}\cdot \mathcal{E}_U$ one has:
\begin{align}
&T_2^{-1}(\mathcal{E}_U) = \{ Q^T\cdot M^{-1}\cdot U + m\cdot Q^T\cdot e_1 | \|U\| \leq 1\} \nonumber \\
& = \left\{ Q^T\cdot M^{-1} \cdot U - m\cdot \frac{\hat{H}_k}{\| \hat{H}_k \|} \biggr | \| U \| \leq 1 \right\} \nonumber \\
& = \left \{ X \biggr | \left(X + m\cdot \frac{\hat{H}}{\|\hat{H}_k\|} \right)^T Q^T \cdot M^T \cdot M \cdot Q \cdot 
\left (X + m\cdot \frac{\hat{H}}{\| \hat{H}_k\|} \right) \leq 1 \cdot \right \}
\end{align} but 
\begin{align}
Q^T \cdot M^T \cdot M \cdot Q & = Q^T \cdot\frac{n^2 - 1}{n^2} \cdot \left( I + \frac{2}{n-1} \cdot e_1 \cdot e_1^T \right) \cdot Q \nonumber \\
& = \frac{n^2 - 1}{n^2} \cdot \left( I + \frac{2}{n-1} \cdot \frac{\hat{H}_k \cdot \hat{H}_k^T}{\| \hat{H}_k\|^2 } \right) \nonumber \\
& = M_k^T \cdot M_k
\end{align} hence 
\begin{align}
T_2^{-1}(\mathcal{E}_U) &= \left \{ X \biggr | \left (X + m \frac{\hat{H}_k}{\| \hat{H}_k\|}\right )^T \cdot M_k^T  M_k \cdot 
\left (X + m \frac{\hat{H}_k}{\| \hat{H}_k\|} \right ) \leq 1 \right \} \nonumber \\
& = \left \{ M_k^{-1} \cdot U - m\cdot \frac{\hat{H}_k}{\| \hat{H}_k \|} \biggr | \| U \| \leq 1 \right \}
\end{align} Finally $T_1^{-1}(T_2^{-1}(\mathcal{E}_U))$ is computed with $T_1^{-1}(Y) = A_k^{-1}\cdot Y + X_k$
\begin{align}
T_1^{-1}(T_2^{-1}(\mathcal{E}_U)) &= \left \{ A_k^{-1} \left (M_k^{-1} \cdot U - m\cdot \frac{\hat{H}_k}{\| \hat{H}_k\|} \right ) 
+ X_k \biggr | \|U\| \leq 1 \right \} \nonumber \\
& = \left \{ \left (M_k A_k \right )^{-1} U - m  A_k^{-2}  \frac{H_k}{\| \hat{H}_k \|} +  X_k \biggr | \| U \| \leq 1 \right \}
\end{align} hence 
\begin{align}
A_{k+1} = M_k\cdot A_k \hspace{1cm} X_{k+1} = X_k - m\cdot A_k^{-2} \cdot \frac{H_k}{\| \hat{H}_k \|}
\end{align} therefore
\begin{align}
A_{k+1}^T\cdot A_{k+1} &= A_k^T\cdot M_k^T\cdot M_k \cdot A_k  \nonumber \\
 & = \frac{n^2-1}{n^2} \left( A_k^T\cdot A_k + \frac{2}{n-1} \cdot \frac{ H_k \cdot H_k^T}{\| \hat{H}_k \|^2} \right) 
\end{align} since $\hat{H}_k = A_k^{-1} \cdot H_k$ and then 
\begin{align}
X_{k+1} = X_k - \frac{1}{n+1} \left( A_k^T\cdot A_k \right)^{-1} \cdot \frac{H_k}{ \| \hat{H}_k^T\|} 
\end{align}

Please note that in both the update formulae one has to divide by $\| \hat{H}_k\| = \| A_k^{-1} \cdot H_k\|$ hence $\| \hat{H}_k\|^2 = 
H_k^T \cdot (A_k^T \cdot A_k)^{-1} \cdot H_k$. Traditionally, the Sherman-Morrison 
formula is used to iteratively obtain $(A_k^{T} \cdot A_k)^{-1}$. Indeed, from \cite{capozzo}

\begin{align}
& \left( A_{k+1}^T \cdot A_{k+1} \right)^{-1} = \nonumber \\
& \frac{n^2}{n^2 - 1} \left( \left(A_k^T\cdot A_k\right)^{-1} - 
\frac{2}{n+1}\frac{ \left( A_k^T\cdot A_k \right)^{-1} \cdot H_k \cdot H_k^T \cdot \left( A_k^T\cdot A_k \right)^{-1} }{\| A_k^{-1}\cdot H_k\|^2}\right)
\end{align}

Now given $\mathcal{E}_{k}$ by $A_k^T\cdot A_k$ and $X_k$ and a hyperplane $H_k$, it is possible to compute an ellipsoid $\mathcal{E}_{k+1}$,  
characterized by $A_{k+1}^T\cdot A_{k+1}$ and $X_{k+1}$, such that $\mathcal{E}_k \cap \mathcal{H}(X_k, H_k) \subseteq \mathcal{E}_{k+1} $ and 
$ \frac{\text{vol}\left(\mathcal{E}_{k+1} \right)}{\text{vol}\left( \mathcal{E}_k \right)} \leq e^{\frac{-1}{2\cdot (n+1)}} < 1$. 

\end{document}